\documentclass[preprint,10pt]{elsarticle}

\usepackage{amsmath}
\usepackage{amsfonts}
\usepackage{amsthm}
\usepackage{amssymb}

\usepackage{algorithmic}

\usepackage{color}
\usepackage{tikz}
\usepackage{url}
\usepackage{bm}

\usepackage[justification=raggedright]{caption}
\usepackage{graphicx}
\usepackage{subcaption}
\usepackage{hyperref}

\newtheorem{theorem}{Theorem}[section]
\newtheorem{corollary}[theorem]{Corollary}

\newtheorem{proposition}[theorem]{Proposition}

\theoremstyle{definition}
\newtheorem{definition}[theorem]{Definition}

\newtheorem{example}[theorem]{Example}

\numberwithin{equation}{section}

\DeclareMathOperator{\supp}{supp}

\def\R{{\mathbb R}}  % The real numbers.
\newcommand{\N}{\mathbb{N}}  % The Natural numbers.
\newcommand{\C}{\mathbb{C}} % The complex numbers.

\newcommand{\q}{\left\{}
\newcommand{\w}{\right\}}
\newcommand{\re}{\mathbb{R}}

\newcommand{\nn}{\mathbb{N}}

\newcommand{\ap}{\alpha}

\newcommand{\lp}{\left(}
\newcommand{\rp}{\right)}
\begin{document}
\begin{frontmatter}


\title{On a construction method of new moment sequences}
%\title{On moment sequences induced by {\color{red} symmetric positive} polynomials}



\author[label1]{Seunghwan Baek}
\ead{seunghwan@knu.ac.kr}
\address[label1]{Department of Mathematics,
Kyungpook National University, Daegu 41566, Republic of Korea}

\author[label1]{Hayoung Choi}
\ead{hayoung.choi@knu.ac.kr}
%\address[label2]{Department of Mathematics,
%Kyungpook National University, Daegu 41566, Korea}

\author[label3]{Seonguk Yoo\corref{cor1}}
\ead{seyoo@gnu.ac.kr}
\cortext[cor1]{Corresponding author}
\address[label3]{Department of Mathematics Education and RINS, Gyeongsang National University, Jinju 52828, Republic of Korea}

\begin{abstract}
In this paper we provide a way to construct new moment sequences from a given moment sequence.
An operator based on multivariate positive polynomials is applied to get the new moment sequences.
A class of new sequences is corresponding to a unique symmetric polynomial; if this polynomial is positive, then the new sequence becomes again a moment sequence.
We will see for instance that a new sequence generated from minors of a Hankel matrix of a Stieltjes moment sequence is also a Stieltjes moment sequence.
\end{abstract}

\begin{keyword}
moment sequence, representing measure, positive semidefinite, totally positive matrix, symmetric positive polynomial.
%A \sep B \sep C
%keywords here, in the form: keyword \sep keyword
%MSC codes here, in the form: \MSC code \sep code
\MSC[2020]{\ Primary 47A57, 44A60; Secondary 15B48, 15A29, 15-04}
\end{keyword}

\end{frontmatter}

%\linenumbers

%\section{Introduction}
%\label{sec:intro}

%\begin{figure}[t!]
%  \centering
%    \includegraphics[width=0.5\textwidth]{fig2.eps}
%     \caption{Taylor-$T$}\label{fig:graph-Taylor-T}
%\end{figure}

\section{Introduction}
Let $\mathbb{N}_0$ be the set of nonnegative  integers and let $\mathbb{R}$ be the set of real numbers.
A nonnegative Borel measure $\mu$ on $\mathbb{R}$ is called a \emph{$K$-measure} if its support, denoted by $\supp(\mu)$, is contained in a closed set $K\subset \mathbb{R}$.
The symbol $\R[x]$ denotes the ring of polynomials in $x$ with real coefficients.
The integral $\int_{K}x^{n} d\mu$, if it exists, is called the \emph{$n$-th moment of the measure $\mu$}.
A sequence $\alpha=(\alpha_n)_{n \in \N_0}$ is said to \emph{admit a $K$-measure $\mu$} if
\begin{equation}\label{eq:moments}
\alpha_{n} = \int_{K} x^{n} d\mu \quad \text{for all }
n \in \N_0.
\end{equation}
Such $\mu$ is called a \emph{$K$-representing measure for $\alpha$} and $\alpha$ is referred to as  a \emph{$K$-moment sequence}. When $K=\R$ (respectively, $K=[0,\infty)$, $K=[0,1]$), the sequence $\alpha$ is also called a \emph{Hamburger} (respectively, \textit{Stieltjes, Hausdorff}) \textit{moment sequence}. A moment sequence is called \emph{determinate}, if it admits  a unique  measure  such that \eqref{eq:moments} holds; otherwise it is called \emph{indeterminate}.
For more information, see  \cite{book:momentproblem} and references therein. %,book:Shohat,book:Akhiezer,Curto97

Given a sequence $\alpha=(\alpha_n)_{n \in \N_0}$, we denote
\begin{equation}\label{eq:Hankelmatrix}
H_m(\alpha):=
\begin{bmatrix}
\alpha_{0} & \alpha_{1} & \cdots & \alpha_{m} \\
\alpha_{1} & \alpha_{2} & \cdots & \alpha_{m+1} \\
\vdots & \vdots & \ddots & \vdots \\
\alpha_{m} & \alpha_{m+1} & \cdots & \alpha_{2m} \\
\end{bmatrix}.
\end{equation}
For a real sequence $\alpha=(\alpha_n)_{n \in \N_0}$
let $E\alpha\equiv E(\alpha)$ denote the shifted sequence; that is,
\begin{equation*}
E\alpha= \lp \alpha_{n+1} \rp_{n\in \N_0} .
\end{equation*}
A necessary and sufficient condition for a moment sequence is following.
\begin{theorem}\label{thm:equivcond}
Let $\alpha$ be a real sequence. Then the followings hold:
\begin{enumerate}[(i)]
\item $\alpha$ is a Hamburger moment sequence if and only if $H_m(\alpha)$ is positive semidefinite for all $m \in \N_0$
\item $\alpha$ is a Stieltjes moment sequence if and only if both $H_m(\alpha)$ and ${H}_m(E\alpha)$ are positive semidefinite for all $m \in \N_0$, equivalently, $H_m(\alpha)$ is totally positive for all  $m \in \N_0$.
\end{enumerate}
\end{theorem}
However, it is in general not easy to check if a given sequence is a Hamburger (or Stieltjes) moment sequence; showing positivity of the Hankel matrix $H_m(\alpha)$ for all nonnegative integer $m$ is usually very difficult.
Sometimes, through simple observation, one can confirm that some new sequences generated from a given moment sequence become the moment sequence again.
For example, if $(\alpha_n)_{n\in \N_0}$ and $(\beta_n)_{n\in \N_0}$ are Hamburger (respectively, Stieltjes) moment sequences, then both $(\alpha_n+\theta\beta_n)_{n\in \N_0}$ and $(\alpha_n\beta_n)_{n\in \N_0}$ are also Hamburger (respectively, Stieltjes) moment sequences for $\theta \geq 0$.
Moreover, if $(\alpha_n)_{n\in \N_0}$ is a Stieltjes moment sequence, then
so are $(\alpha_{n+1})_{n\in \N_0}$ and $(\alpha_{n+2})_{n\in \N_0}$; in particular
 $(\alpha_{n}\alpha_{n+2}+\alpha_{n+1}^2)_{n\in \N_0}$ is a Stieltjes moment sequence.
On the other hand, it is quiet non-trivial to check if $(\alpha_{n}\alpha_{n+2}-\alpha_{n+1}^2)_{n\in \N_0}$ is a Stieltjes moment sequence.

Based on given moment sequences, constructing new moment sequences via linear transformations and convolutions  has been studied in \cite{Ben11, Zhu16}. Moreover, Zhu provided below an interesting result:

\begin{theorem}[\cite{Zhu19}]\label{thm:L2}
For fixed $r,s\in \N$, if $\alpha=(\ap_n)_{n\in \N_0}$ is a Stieltjes moment sequence, then so is
$$\left( \det
\begin{bmatrix}
\alpha_n & \alpha_{n+r}\\
\alpha_{n+s} & \alpha_{n+r+s}
\end{bmatrix}
\right)_{n\in \N_0}
=
(\alpha_{n} \alpha_{n+r+s} - \alpha_{n+r}\alpha_{n+s})_{n\in \N_0}.
$$
\end{theorem}

This result says that the determinants of any $2\times2$ submatrices of $H_m(\ap)$ can be considered as an operator to generate new Stieltjes moment sequences.
In the sequel, we will see this operator corresponds to a symmetric positive polynomials.
We then show more generally that any minor of $H_m$ can be a new Stieltjes moment sequence by checking positivity of the associated symmetric polynomial to the minor.
We will also provide several operators to generate new Stieltjes (or Hausdorff) moment sequences.

This paper consists of four sections. In Section 2, we provide a way to construct new sequences from a given moment sequences by polynomials and we show that if such polynomial is positive, then the induced sequence is a moment sequence. In Section 3, we discuss a symmetric polynomial generating new sequence and their properties. Moreover, we provide a generalization of Theorem \ref{thm:L2}. In Section 4, we also obtain similar results for Hausdorff moment sequence.

\section{New sequences and positive polynomials}

In this section we will use multisequences which admits a $K$-measure for $K\subset \mathbb{R}^d$ and positive multivariate polynomials to generate new moment sequence from given moment sequences. Let us first introduce a $K$-moment problem for $K\subset \mathbb{R}^d$.

Let $\beta \equiv ( \beta_{\pmb i})_{{\pmb i} \in \mathbb{N}_{0}^d}$ denote
a $d$-dimensional real multisequence and let $K$ denote a closed subset of $\mathbb{R}^d$. The full $K$-moment problem asks for conditions on $\beta$ such that there exists a positive Borel measure $\mu$, with $\supp (\mu) \subset K$, satisfying
\begin{equation*}
\beta_{\pmb i} = \int_K {\pmb x}^{\pmb i} d\mu \quad \text{for all }{\pmb i}\in \mathbb{N}_{0}^d ,
\end{equation*}
where  ${\pmb x}^{\pmb i} :=x_1^{i_1}\cdots x_d^{i_d}$ for ${\pmb x}\equiv (x_1,\ldots, x_d)\in \mathbb{R}^d$ and ${\pmb i} \equiv (i_1,\ldots, i_d) \in \mathbb{N}_{0}^d$.

Let $\mathbb{R}[{\pmb x}]=\mathbb{R}[x_1,\ldots,x_d]$ and let $p \equiv p(\pmb x)=\sum p_{\pmb i} {\pmb x}^{\pmb i} \in \mathbb{R}[{\pmb x}]$.
Corresponding to $\beta$, the \emph{Riesz functional} $\mathcal{L}\equiv \mathcal{L}_{\beta} :\mathbb{R}[{\pmb x}] \to \mathbb{R}$ is defined by
\begin{equation*}
\mathcal{L}\left( \sum p_{\pmb i} {\pmb x}^{\pmb i}  \right) := \sum p_{\pmb i} \beta_{\pmb i}
\end{equation*}
We say that $\mathcal{L}$ is \emph{$K$-positive} if
\begin{equation*}
\mathcal{L}_{\beta} (p) \geq 0 \quad \forall p\in \mathbb{R}[{\pmb x}]: p|_K \geq 0.
\end{equation*}
When $K=\R^d$, we call simply that $\mathcal{L}$ is \emph{positive} instead of $K$-positive. The $K$-positivity of $\mathcal{L}_{\beta}$ is a necessary condition for $\beta$ to admit a $K$-measure. Conversely, the classical theorem of Riesz and Haviland provides a fundamental existence criterion for $K$-representing measures.

\begin{theorem}[Riesz-Haviland Theorem]\label{thm:RieszHavil}
A $d$-dimensional real multisequence $\beta$ admits a representing measure supported in the closed set $K\subset \mathbb{R}^d$ if and only if $\mathcal{L}_{\beta}$ is $K$-positive.
\end{theorem}

Using the Riesz-Haviland Theorem, we can prove that some new sequences generated with  terms in a   Stieltjes  moment sequence are also Stieltjes  moment sequences.
We denote $\mathcal{A}$ for the set of all real sequences.

\begin{definition}\label{def-ptos}
For a given polynomial $p=\sum p_{\pmb i} {\pmb x}^{\pmb i} \in \mathbb{R}[\pmb x]$, we define $T_p:\mathcal{A}\to \mathcal{A}$ given by
\begin{equation}\label{def:Tp}
  T_p(\alpha)=\tilde{\alpha}=(\tilde{\alpha}_n)_{n\in \N_0} \quad \text{for all }\alpha\in \mathcal{A},
\end{equation}
where
$$\tilde{\alpha}_n = \sum p_{\pmb i}\alpha_{n+i_1}\alpha_{n+i_2}\cdots \alpha_{n+i_d} \ \ \text{for all} \ \ n \in \N_0.$$
\end{definition}

For example, when $p=\sum p_i x^i$ is a one-variable polynomial, $\tilde{\alpha}_n$ is defined as
$$\tilde{\alpha}_n =  \sum p_i \alpha_{n+i}.$$

\begin{definition}
Let $K$ be a closed subset of $\mathbb{R}$ and $p\in \mathbb{R}[\pmb x]$.
It is said that \emph{$T_p$ has $K$-moment property} if for any $K$-moment sequence $\alpha$, $T_p(\alpha)$ is $K$-moment sequence.
When $K=[0,\infty)$ (respectively, $\R$, $[0,1]$), we simply say that $T_p$ has \emph{Stieltjes} (respectively, \textit{Hamburger, Hausdorff}) \textit{moment property}.
\end{definition}

The following is one of the main results.

\begin{theorem}[$d=1$]\label{thm:mainSM1d}
Let $p \in \mathbb{R}[x]$ be a one-variable polynomial. Then $T_p$ has the Stieltjes moment property if and only if
 %A necessary and sufficient condition that $T_p$ has the Stieltjes moment property is
$$p|_{[0,\infty)} \geq 0.$$
%If \textcolor{red}{$p|_{K^d} \geq 0$}, then $T_p$ has the Stieltjes moment property.
\end{theorem}

\begin{proof}
($\Leftarrow$) Let $\alpha = ( \ap_n )_{n \in \N_0}$ be a Stieltjes moment sequence. Suppose that $p=\sum p_{i} {x}^{i}$ is nonnegative on $[0, \infty)$ and $\tilde{\alpha}=T_p(\alpha)$ is defined as in \eqref{def:Tp}.
Then by Theorem \ref{thm:RieszHavil} it is enough to show that
$$\mathcal{L}_{\tilde{\alpha}}(q) \geq 0 \quad \forall q\in \mathbb{R}[x]: q|_{[0,\infty)} \geq 0.$$
Let  $q=\sum q_{i} {x}^{i}\in \mathbb{R}[x] $ be nonnegative on $[0, \infty)$.
Then it follows that
\begin{align*}
\mathcal{L}_{\tilde{\alpha}}(q(x))
=\sum_{i} q_i \tilde{\alpha}_i
=\sum_{i} \sum_{j} p_j q_i \alpha_{i+j}
=\mathcal{L}_{\alpha}(p(x)q(x))\geq 0.
\end{align*}
The last inequality holds since both $p$ and $q$ are nonnegative on $[0, \infty)$.
\smallskip

\noindent($\Rightarrow$) Suppose that $T_p$ has the Stieltjes moment property and there exists $\xi\in (0,\infty)$ such that $p(\xi) <0$.
Set $d\mu = \delta_{\xi}$ which is a dirac measure.
Let $\alpha=(\alpha_n)_{n\in\mathbb{N}_0}$ defined by
$$\alpha_n =\int_0^{\infty} x^n d\mu.$$
Then $\alpha$ is a Stieltjes moment sequence. Observe that if  $p=\sum p_{i} {x}^{i}$, then 
$$\tilde{\alpha}_0 = \sum p_i \alpha_i = \sum p_i \xi^i = p(\xi)< 0.$$
However, since $\tilde{\alpha}=T_p(\alpha)$ is a Stieltjes moment sequence, it follows that $\tilde{\alpha}_0 \geq 0$. This is a contradiction.
\end{proof}

\begin{example}
Suppose $\alpha=(\alpha_n)_{n\in \N_0}$ is a Stieltjes moment sequence.
\begin{enumerate}[(i)]
  \item For fixed $k \in \N$, consider a new sequence $\tilde{\alpha}=(\tilde{\alpha}_n)_{n\in \N_0}$ defined by $$\tilde{\alpha}_n=c_0\alpha_{n+k}+c_1\alpha_{n+k+1}+c_2\alpha_{n+k+2} \ \ \text{for} \ \ c_0,c_1,c_2 \in \R.$$
Actually, it is difficult to show that $\tilde{\alpha}$ is a Stieltjes moment sequence by checking the positivity of Hankel determinant of $\tilde{\alpha}$ (Sometimes, for a case like the Catalan number $\alpha_n$, we can find the exact Hankel determinant of $\tilde{\alpha}_n$ (\cite{Bou}), but such a case is rare.)
However, we can obtain a simple equivalent condition by using Theorem \ref{thm:mainSM1d}.
      In fact, $T_p(\alpha)=\tilde{\alpha}$, where $p(x)=x^k(c_0+c_1x+c_2x^2)$. If you want to know a sufficient and necessary condition that $\tilde{\alpha}$ is a Stieltjes moment sequence, it is enough to verify the quadratic polynomial $c_0+c_1x+c_2x^2\geq0$ on $[0,\infty)$; it is easy to obtain this condition.
  \item We know that $T_p(\alpha)$ is a Stieltjes moment sequence if and only if $p(x)\geq0$ on $[0,\infty)$, it is equivalent from \cite[Corollary 3.25]{book:momentproblem} that $p(x)=[f(x)]^2+x[g(x)]^2$ for some $f(x)=\sum f_i x^i$ and $g(x)=\sum g_j x^j \in \mathbb{R}[x]$. i.e., $\tilde{\alpha}\equiv(\tilde{\alpha}_n)_{n\in\N_0}$ is a Stieltjes moment sequence, where
    $$\tilde{\alpha}_n=\sum_{i,k} f_{i}f_ka_{n+i+k}+\sum_{j,\ell} g_{j}g_{\ell}a_{n+j+\ell+1}.$$
\end{enumerate}
\end{example}

Now we extend Theorem \ref{thm:mainSM1d}.

\begin{theorem}[$d\geq2$]\label{thm:mainSM}
%Let \textcolor{red}{$K=[0,\infty)$} and
Let $p \in \mathbb{R}[{\pmb x}]=\mathbb{R}[x_1,\ldots,x_d]$ with $d\geq 2$.
Then $T_p$ has the Stieltjes moment property if
$$p|_{[0,\infty)^d} \geq 0.$$
%If \textcolor{red}{$p|_{K^d} \geq 0$}, then $T_p$ has the Stieltjes moment property.
\end{theorem}

\begin{proof}
Let $\alpha = ( \ap_n )_{n \in \N_0}$ be a Stieltjes moment sequence.
Then there exists a positive Borel measure $\mu$ with $\supp (\mu) \subset [0,\infty)$ such that
\begin{equation*}
\alpha_n= \int_0^\infty x^n d\mu(x) \quad \text{for all }n\in \mathbb{N}_0.
\end{equation*}
Suppose that $p=\sum p_{\pmb i} {\pmb x}^{\pmb i}$ is nonnegative on $[0, \infty)^d$ and $\tilde{\alpha}=T_p(\alpha)$ is defined as in \eqref{def:Tp}.
Then it follows that for each $m\in \N_0$ the quadratic form
\begin{align*}
\sum_{j=0}^{m} \sum_{k=0}^{m} \tilde{\alpha}_{j+k} \xi_{j} \xi_{k}
&=
\sum_{j=0}^{m} \sum_{k=0}^{m} \sum_{{\pmb i}} p_{\pmb i} \alpha_{j+k+i_1} \cdots \alpha_{j+k+i_d} \xi_{j} \xi_{k}\\
&=
\sum_{j=0}^{m} \sum_{k=0}^{m} \sum_{{\pmb i}}  p_{\pmb i} \left(\prod_{\ell=1}^{d} \int_0^\infty x_{\ell}^{j+k+i_\ell} d\mu(x_\ell) \right) \xi_{j} \xi_{k}\\
&=
\int_0^\infty \cdots \int_0^\infty \sum_{j=0}^{m} \sum_{k=0}^{m}  (x_1\cdots x_d)^{j+k}\sum_{{\pmb i}} p_{\pmb i}   {\pmb x}^{\pmb i} \xi_{j} \xi_{k} d\mu \cdots d\mu\\
&=
\int_0^\infty \cdots \int_0^\infty \sum_{j=0}^{m} \sum_{k=0}^{m}  (x_1\cdots x_d)^{j+k} \xi_{j} \xi_{k}p({\pmb x})  d\mu \cdots d\mu \\
&=
\int_0^\infty \cdots \int_0^\infty \left( \sum_{j=0}^{m}   (x_1\cdots x_d)^{j} \xi_{j} \right)^2  p({\pmb x})  d\mu \cdots d\mu
\end{align*}
is nonnegative for all $\xi_0,\ldots, \xi_m \in \mathbb{C}$.
In the similar way, one can check that
\begin{align*}
&\sum_{j=0}^{m} \sum_{k=0}^{m} \tilde{\alpha}_{j+k+1} \xi_{j} \xi_{k}\\
&\ \ \ \ =  \int_0^\infty \cdots \int_0^\infty \left( \sum_{j=0}^{m}   (x_1\cdots x_d)^{j} \xi_{j} \right)^2  (x_1 \cdots x_d) p({\pmb x})  d\mu \cdots d\mu
\end{align*}
is nonnegative for each $m\in \mathbb{N}_0$.
Thus  $\tilde{\alpha}$ is also a Stieltjes moment sequence.
\end{proof}

The proof of Theorem \ref{thm:mainSM} also implies the following.

\begin{corollary}\label{coro:HM}
If $p \in \mathbb{R}[\bm{x}]$ is nonnegative on $\mathbb{R}^d$, then $T_p$ has the Hamburger moment property.
\end{corollary}

The next examples illustrate how Theorem \ref{thm:mainSM} works.

\begin{example} \label{exa:2.2}
Suppose $\alpha=(\alpha_n)_{n\in \N_0}$ is a Stieltjes moment sequence.

\begin{enumerate}[(i)]
\item  The polynomial $p(x,y)=\frac{1}{2} (x^r - y^r)(x^s - y^s)$ is nonnegative on   $[0, \infty)^2$ for fixed $r,s\in \nn$ and one can easily  check that
\begin{equation}\label{eq:a22}
  T_p(\ap)=\left( \det
\begin{bmatrix}
\alpha_n & \alpha_{n+r} \\
\alpha_{n+s} & \alpha_{n+r+s}
\end{bmatrix}
 \right)_{n\in \N_0}.
\end{equation}
    It follows from Theorem \ref{thm:mainSM} that $T_p$ has the Stieltjes moment property; a very simple verification of Theorem \ref{thm:L2} is completed just now.

\item Obviously, the polynomial $ p(x,y,z)= \frac{1}{6}(x-y)^2 (y-z)^2  (z-x)^2$ is nonnegative on $[0,\infty)^3$ and a calculation shows $p$ corresponds to the new sequence
\begin{equation}\label{eq:33}
T_p(\alpha)=\left( \det
\begin{bmatrix}
\alpha_n & \alpha_{n+1} & \alpha_{n+2} \\
\alpha_{n+1} & \alpha_{n+2} & \alpha_{n+3} \\
\alpha_{n+2} & \alpha_{n+3} & \alpha_{n+4}
\end{bmatrix}
 \right)_{n\in \N_0}.
\end{equation}
Applying Theorem \ref{thm:mainSM}, we know that $T_p(\alpha)$ is a Stieltjes moment sequence.

\item Through the same discussion, it can be seen that all new sequences corresponding to the following cases become  Stieltjes moment sequences:
\begin{align*}
T_{p_1} (\ap)  & = \left(\ap_{n+2} \ap_{n+4} - \ap_{n+2}^2 + \ap_{n}^2\right)_{n \in \N_0},\\
T_{p_2} (\ap)& =  \left(\frac{(-1)^{r}}{2} \binom{2r}{r} \ap_{n+r}^2+  \sum_{k=0}^{r-1} (-1)^{k} \binom{2r}{k} \ap_{n+k}\ap_{n+2r-k}\right)_{n \in \N_0},\\
T_{p_3} (\ap)& = \left(\ap_{n}^2 \ap_{n+3} - \ap_{n+1}^3\right)_{n \in \N_0},
\end{align*}
where $p_1(x,y)= {\frac{1}{2} (x^4y^2 + x^2y^4 - 2x^2y^2 + 2)}$,  $p_2(x,y)={\frac{1}{2}(x-y)^{2r}}$, and $p_3(x,y,z)={\frac{1}{3} (x^3+y^3 + z^3 -3xyz)}$.
\end{enumerate}
\end{example}

Note that it is highly nontrivial to show if the sequences in Example \ref{exa:2.2} are Hamburger (or Stieltjes) moment sequences; showing positivity of the Hankel matrix $H_m(\alpha)$ for all nonnegative integer $m$ is not simple.

Notice that most preceding polynomials so far are homogeneous but this is not always the case like  $p_1$ in Example \ref{exa:2.2} (i).
We will later define a sort of an inverse operator of  Definition \ref{def-ptos};  a class of new sequences will be considered and its members are supposed to correspond to a unique homogeneous polynomial.

Next, we briefly discuss a new moment sequence generated from a quadratic form.
Recall that a $d\times d$ symmetric real matrix $A$ is \emph{copositive} (respectively, \textit{positive semidefinite}) if
${\pmb x}^T A {\pmb x} \geq 0$ for all vectors ${\pmb x}\in [0,\infty)^d$ (respectively,  ${\pmb x}\in \mathbb{R}^d$).
Then it is trivial that if $A$ is copositive (respectively, positive semidefinite), then $p({\pmb x}):={\pmb x}^T A {\pmb  x}\in \mathbb{R}[{\pmb x}]$ is nonnegative on $[0,\infty)^d$ (respectively, $\mathbb{R}^d$).
Hence, from Theorem \ref{thm:mainSM} and Corollary \ref{coro:HM} we get the following.

\begin{corollary}\label{corollary:copositive-new}
Let $A=[a_{ij}]_{i,j=1}^d$ be a $d\times d$ symmetric copositive (respectively, positive semidefinite) matrix and
$\tilde{\alpha}=( \tilde{\alpha}_n)_{n\in \N_0}$ be a sequence defined as
\begin{equation*}\label{eq:copositive-new}
 \tilde{\alpha}_n = \sum_{i=1}^d a_{ii} \alpha_{n+2} \alpha_{n}^{d-1} + \sum_{i\neq j} a_{ij} \alpha_{n+1}^2 \alpha_n^{d-2}.
 \end{equation*}
If $(\alpha_n)_{n\in \N_0}$ is a Stieltjes (respectively, Hamburger) moment sequence, then
so is $(\tilde{\alpha}_n)_{n\in \N_0}$.
\end{corollary}

\begin{example}\label{example:2by2det}
One can show that  $A=\left[
         \begin{array}{cc}
           1/2 & -1/2 \\
          -1/2 & 1/2 \\
         \end{array}
       \right]$
is a copositive (respectively, positive semidefinite) matrix. Then
Corollary \ref{corollary:copositive-new} implies that
if $(\alpha_n)_{n\in \N_0}$ is a Stieltjes (respectively, Hamburger) moment sequence, then
so is
\begin{equation}\label{eq:22}
\left( \det
\begin{bmatrix}
\alpha_n & \alpha_{n+1}\\
\alpha_{n+1} & \alpha_{n+2}
\end{bmatrix}
\right)_{n\in \N_0}.
\end{equation}
\end{example}

Wang and Zhu (\cite{Zhu16}) proved that \eqref{eq:22} is a Stieltjes moment sequence for any Stieltjes moment sequence $(\alpha_n)_{n\in\N_0}$, based on the positivity of compound matrix of the Hankel matrix in \eqref{eq:Hankelmatrix}.
We simply showed that it is a Stieltjes moment sequence in a unified approach.
Moreover, \eqref{eq:a22} and \eqref{eq:33} can be considered as extensions of \eqref{eq:22}.
Now we show more generalized case; all new sequences generated from any minors of the Hankel matrix $H_m (\ap)$ are to have the Stieltjes moment property.

\begin{theorem}\label{thm:Ln}
Let $0<r_1<r_2<\cdots<r_{d-1}$ and $0<t_1<t_2<\cdots<t_{d-1}$ be natural numbers in a given order. If $\alpha=(\alpha_n)_{n\in \N_0}$ is a Stieltjes moment sequence, then so is
\begin{equation}\label{ddd}
\left(\det \left[
         \begin{array}{cccc}
           \alpha_{n} & \alpha_{n+r_1} & \cdots & \alpha_{n+r_{d-1}} \\
           \alpha_{n+t_1} & \alpha_{n+r_1+t_1} & \cdots & \alpha_{n+r_{d-1}+t_1} \\
           \vdots & \vdots & \ddots & \vdots \\
           \alpha_{n+t_{d-1}} & \alpha_{n+r_1+t_{d-1}} & \cdots & \alpha_{n+r_{d-1}+t_{d-1}} \\
         \end{array}
       \right ]\right)_{n\in \N_0}.
\end{equation}
\end{theorem}

To prove the theorem, we need to know some properties of symmetric positive polynomials. The proof will be shown in the next section.

\section{Symmetric positive polynomials}

In previous section, we can see that for a given Stieltjes moment sequence and a positive polynomial, the associated sequence is also a Stieltjes moment sequence. From the opposite point of view, the following question arises naturally;
\begin{center}
\textit{for any given $\beta \in\mathcal{A}$, are there $p \in \R[\pmb x]$ and $\alpha\in\mathcal{A}$ such that $T_p(\alpha)=\beta$?}
\end{center}
If so, we can know that $\beta$ is a Stieltjes moment sequence whenever $\alpha$ is a Stieltjes moment sequence and $p$ is nonnegative on $[0,\infty)^d$.
In fact, we do not know the existence of $\alpha$ and $p$, but once they exist, there can be infinitely many such polynomials (not necessarily symmetric or positive). For example, let $\beta=(\beta_n)_{n\in\N_0}$ be a real sequence with $\beta_n=\alpha_n\alpha_{n+2}-\alpha_{n+1}^2$ ($n\in\N_0$) for some Stieltjes moment sequence $\alpha=(\alpha_n)_{n\in\N_0}$. Then $T_p(\alpha)=\beta$, where
 $$p_\theta(x,y)=\frac{1-\theta}{2}x^2-xy+\frac{1+\theta}{2}y^2 \ \ \text{for all} \ \ \theta \in \R.$$
That is, we can see that there are infinitely many polynomials corresponding to the new sequence. But for $\theta\neq0$, $p_\theta$ is not nonnegative on $[0,\infty)^2$, so Theorem \ref{thm:mainSM} cannot be applied for this case. However, since $p_0(x,y)$ is nonnegative on $[0,\infty)^2$, $T_{p_\theta}(\alpha)$ is a Stieltjes moment sequence. We also observe that $p_0(x,y)=\frac{1}{2}(x-y)^2$ is symmetric (see the definition below). Therefore, if we want to know that a given sequence $\beta$ is moment sequence, we need to find a corresponding positive polynomial, especially, a symmetric positive polynomial.
 In this section we study symmetric positive polynomials and prove Theorem \ref{thm:Ln} by properties of symmetric positive polynomials.

Denote $\mathcal{S}$ for the set of all permutations on $\{1,\ldots,d\}$. A multivariate polynomial $p(\pmb x) \in \R[\pmb x]$ is called \textit{symmetric} if
$$p(x_1,\ldots,x_d)=p(x_{\pi(1)},\ldots,x_{\pi(d)}) \ \ \text{for all} \ \ \pi\in \mathcal{S}.$$

%Denote $\mathcal P$ for the set of all symmetric polynomials in $\re[\pmb x]$.
%We define a kind of an inverse map in Definition \ref{def-ptos}:

%\begin{definition} \label{def:sp} Given a real sequence $\ap=\lp \alpha_n \rp_{n \in \N_0} \in \mathcal{A}$,  denote a set
%$$ \tilde{\mathcal A}:= \left\{ \tilde{\alpha}\equiv(\tilde{\alpha}_n)_{n \in \N_0}:\tilde{\alpha}_n = \sum_{{\pmb i} \in \N_0^d} p_{\pmb i}\alpha_{n+i_1}\alpha_{n+i_2}\cdots \alpha_{n+i_d}, \  p_{\pmb i} \in \R  \right\}.$$
%Define a linear map $S:  \tilde{\mathcal A} \ra \mathcal P$ by
%$$S\lp \tilde{\alpha} \rp = \frac{1}{d!} \sum_{\pi\in \mathcal{S}}  p_{\pmb i}    x_{\pi(1)}^{i_1} \cdots x_{\pi(d)}^{i_d}$$
%for each $\tilde{\alpha}=\lp \tilde{\alpha}_n \rp_{n \in \N_0} \in \tilde{\mathcal A}$ with $\tilde{\alpha}_n = \sum_{{\pmb i} \in \N_0^d} p_{\pmb i}\alpha_{n+i_1}\alpha_{n+i_2}\cdots \alpha_{n+i_d}$.
%\end{definition}

%We can see that for given real sequence $\alpha \in \mathcal{A}$,
%\begin{equation*}
%  T_{S(\tilde{\alpha})}(\alpha)=\tilde{\alpha} \ \ \text{for all} \ \ \tilde{\alpha} \in \tilde{\mathcal A}.
%\end{equation*}

For any given real sequence $\beta$, if we  know there are $\alpha$ and $p$ such that $T_p(\alpha)=\beta$, we have questions:
\begin{enumerate}[(i)]
  \item Can we find such $p$ and $\alpha$?
  \item Is there a symmetric polynomial $q$ such that $T_p(\alpha)=\beta=T_q(\alpha)$?
  \item If so, is such a symmetric polynomial unique?
  \item If so, how can we obtain such a symmetric polynomial?
\end{enumerate}
For the question (i), it is difficult to suggest in general how to obtain such $p$ and $\alpha$.
For questions (ii) and (iii), the answers are yes.
We will provide an answer to the question (iv) below.

\begin{definition}
 For a given polynomial $p \in \mathbb{R}[\pmb x]$
define the polynomial $\overline{p} \in \mathbb{R}[\pmb x]$ by
  \begin{equation*}
    \overline{p}(x_1,\ldots,x_d)=\frac{1}{d!}\sum_{\pi\in \mathcal{S}}p(x_{\pi(1)},\ldots,x_{\pi(d)}).
  \end{equation*}
\end{definition}

Then we have some properties of $\overline p$, which are easy to show, so their proofs are omitted.

\begin{proposition}\label{thm:pos}
  Let $p,q \in \mathbb{R}[\pmb x]$ and $K$ be a closed subset of $\R$.
  Then the followings hold on $K^d$:
  \begin{enumerate}[(i)]
  \item $\overline{p}$ is symmetric.
  \item $\overline{p+q} = \overline{p}+\overline{q}$.
  \item $\overline{cp} = c\overline{p}$, where $c$ is constant.
  \item $\overline{p+c} = \overline{p}+c$, where $c$ is constant.
  \item If $p$ is symmetric, then $p=\overline{p}$.
  \item If $p$ is nonnegative, then so is $\overline{p}$.
  \item $T_p(\alpha)=T_{\overline{p}}(\alpha)$ for any sequence $\ap \in \mathcal{A}$.
  \item If $T_p(\alpha)=T_q(\alpha)$, then $\overline{p}=\overline{q}$ for any sequence $\ap \in \mathcal{A}$.
  \end{enumerate}
\end{proposition}

For a given real sequence $\beta$, we suppose that there are a sequence $\alpha$ and a polynomial $p$ such that $T_p(\alpha)=\beta$. Then it follows from Proposition \ref{thm:pos} that although $\beta$ can have infinitely many such polynomials, such a symmetric polynomial must be unique.
So even though many polynomials are not nonnegative, a symmetric polynomial $\overline{p}$ can possibly be nonnegative.
For an example, consider a polynomial $p(x,y)=3x^2-y^2$, which is not nonnegative on $\mathbb{R}^2$, but its corresponding symmetric polynomial $\overline{p}(x,y)=x^2+y^2$ is nonnegative on $\mathbb{R}^2$.
We also have another example appeared at the beginning of this section; $\overline{p_\theta}=p_0$ for all $\theta\in\R$.

We now have a criterion; by checking the positivity of symmetric polynomial, we can know that a associated sequence is a Stieltjes moment sequence or not.
Incidentally, if a corresponding polynomial is not nonnegative, it is inconclusive, that is, $T_p$ can possibly have the Stieltjes moment property. Let  $\ap_n = 2^n$ and consider  $\tilde\ap_n = \ap_n  \ap_{n+2} - t   \ap_{n}\ap_{n+1} = 2^{2n+1}(2-t)$ for $t\in\re$. If $0<t<2$, then $\tilde \ap_n$ is a Stieltjes moment sequence.  We next find a symmetric polynomial $\overline{p}=\frac{1}{2}(x^2 + y^2- t x-t y)$ and observe that this polynomial is not positive on $[0,\infty)^2$ for $t>0$.

\begin{corollary}\label{coro:pos}
 Let $p\in \mathbb{R}[\pmb x]$. If $\overline{p}$ is nonnegative on $[0,\infty)^d$, then $T_p$ has Stieltjes moment property.
\end{corollary}

Here, using the properties for symmetric positive polynomials, we prove Theorem \ref{thm:Ln} which is the result of generalizing Theorem \ref{thm:L2}.

\begin{proof}[Proof of Theorem \ref{thm:Ln}]
  Let $\tilde{\alpha}$ be a sequence as in \eqref{ddd} and let
  \begin{equation*}
    p(x_1,\ldots,x_d):=\det \left[
         \begin{array}{cccc}
           1 & x_1^{r_1} & \cdots & x_1^{r_{d-1}} \\
           x_2^{t_1} & x_2^{r_1+t_1} & \cdots & x_2^{r_{d-1}+t_1} \\
           \vdots & \vdots & \ddots & \vdots \\
           x_d^{t_{d-1}} & x_d^{r_1+t_{d-1}} & \cdots & x_d^{r_{d-1}+t_{n-1}} \\
         \end{array}
       \right ].
  \end{equation*}
  Then we have $T_p(\alpha)=\tilde{\alpha}$. For any finite sequence $\beta=(b_n)_{n=1}^{d-1}$, denote the generalized Vandermonde determinant by
  \begin{equation*}
    D_\beta(x_1,\ldots,x_d):= \det \left[
         \begin{array}{cccc}
           1 & x_1^{b_1} & \cdots & x_1^{b_{d-1}} \\
           1 & x_2^{b_1} & \cdots & x_2^{b_{d-1}} \\
           \vdots & \vdots & \ddots & \vdots \\
           1 & x_d^{b_1} & \cdots & x_d^{b_{d-1}} \\
         \end{array}
       \right ].
  \end{equation*}
Then we have
  \begin{equation*}
    p(x_1,\ldots,x_d) = \prod_{i=1}^{d}x_i^{t_{i-1}} D_\rho(x_1,\ldots,x_d),
  \end{equation*}
  where $t_0=0$ and $\rho:=(r_n)_{n=1}^{d-1}$, it follows that
  \begin{align*}
    \overline{p}(x_1,\ldots,x_d) &= \frac{1}{d!} \sum_{\sigma \in \mathcal{S}} \prod_{i=1}^{d}x_{\sigma(i)}^{t_{i-1}} D_\rho(x_{\sigma(1)},\ldots,x_{\sigma(d)})\\
    &= \frac{1}{d!} \sum_{\sigma \in \mathcal{S}} \prod_{i=1}^{d}x_{\sigma(i)}^{t_{i-1}} \cdot \text{sgn}(\sigma) \cdot D_\rho(x_1,\ldots,x_d)\\
    &= \frac{1}{d!} D_\rho(x_1,\ldots,x_d) \cdot \det \left[ x_j^{t_{i-1}} \right]_{i,j=1}^d \\
    &= \frac{1}{d!} D_\rho(x_1,\ldots,x_d) \cdot D_\tau(x_1,\ldots,x_d),
  \end{align*}
   where $\tau:=(t_n)_{n=1}^{d-1}$. If $x_i=x_j$ for some $i\neq j$ or $x_k=0$ for some $k$, then $\overline{p}(x_1,\ldots,x_d)=0$. Let $x_1,\ldots,x_d$ be any mutually distinct positive real numbers. Then there is a permutation $\pi \in \mathcal{S}$ such that $x_{\pi(i)}<x_{\pi(i+1)}$ for $i=1,\ldots,d-1$. By some determinant properties, we have
   \begin{equation*}
     D_\rho(x_1,\ldots,x_d) \cdot D_\tau(x_1,\ldots,x_d)=D_\rho(x_{\pi(1)},\ldots,x_{\pi(d)}) \cdot D_\tau(x_{\pi(1)},\ldots,x_{\pi(d)}).
   \end{equation*}
It follows from \cite[Theorem 1]{YWZ}, $D_\rho(x_{\pi(1)},\cdots,x_{\pi(d)})$ and $D_\tau(x_{\pi(1)},\cdots,x_{\pi(d)})$ are positive. Hence, $\overline{p}$ is nonnegative on $[0,\infty)^d$. By Corollary \ref{coro:pos}, $T_p$ has the Stieltjes moment property. That is, $\tilde{\alpha}$ is a Stieltjes moment sequence.
\end{proof}

A new moment sequence $\tilde \ap$ must  have a unique corresponding symmetric polynomial; if it is nonnegative, then  $\tilde \ap$ has the Stieltjes and Hamburger moment properties.
A criterion that can determine which  symmetric polynomials are positive is essential.
That is, it is necessary to confirm that the minimum  of a symmetric polynomial is nonnegative. The first approach might be to solve a system of symmetric polynomial equations of first derivatives to find the critical point of a given symmetric polynomial. The second  approach is to directly check  whether a symmetric polynomial belongs to the cone of positive polynomials.
Recall that there are significant many polynomials that are generally positive, but not expressed as sum of squares.
In any case, it is difficult to use these two methods in practice  and one should be familiar with significant  background knowledge of algebraic geometry (see \cite{BlRi21, Las10}).
We instead would like to  introduce a relatively intuitive  result which is the so-called half-order principle.
For $k \in \nn$, let $A_k$ be the set of all points in $\re^d$ with at most $k$ distinct components, that is,
$$
A_k := \q {\pmb x}\equiv (x_1, \ldots, x_d) \in \re^d : | \q x_1, \ldots, x_d \w |\le k  \w.
$$

\begin{proposition}[\cite{BlRi21}, Half degree principle]\label{hdp}
Let $p$ be a symmetric homogeneous polynomial  of degree $2m$ in $d$ variables.
and set $k:=\max \q 2, \, m\w$. Then $f$ is nonnegative if and only if
$p(y)\ge 0 \text{ for all }y\in A_k.$
\end{proposition}

Incidentally, we would like to discuss more about the optimization of the symmetric polynomial $\overline p$.

\begin{theorem}\label{thm:avef}
Let $p\in \mathbb{R}[\pmb x]$ and $K$ be a closed subset of $\R$.
  If $p$ has a minimum on $K^d$, then $\underset{{\pmb x}\in K^d}{\min}\, p \leq \underset{{\pmb x}\in K^d}{\min}\, \overline{p}$.
\end{theorem}

\begin{proof}
  Since $p$ has a minimum,
%  $p(x_1,\cdots,x_d) \geq p(\mathbf{x}_0)$ for some $\mathbf{x}_0 \in \mathbb{R}^d$,
there exists $\pmb{x}^*\in K^d$ such that $\pmb{x}^*=\underset{\pmb{x}\in K^d}{\arg \min}\, p(\pmb{x})$.
By Proposition \ref{thm:pos}, $\overline{p}-p(\pmb{x}^*)=\overline{p-p(\pmb{x}^*)} \geq 0$, and then $\overline{p}(\pmb{x}) \geq p(\pmb{x}^*)$ for all $\pmb{x}\in K^d$.
\end{proof}

\begin{theorem}
 Let $p\in \mathbb{R}[\pmb x]$ and $K$ be a closed subset of $\R$. If $p$ has a minimum at $\pmb {a}=(a_1,a_1,\ldots,a_1)\in K^d$, then $\underset{{\pmb x}\in K^d}{\min} \,p = \underset{{\pmb x}\in K^d}{\min} \,\overline{p}$ and $\overline{p}$ has a minimum at $\pmb {a}$.
%  \textcolor{red}{What is $K$? It is not clear!!!}
\end{theorem}

\begin{proof}
  It suffices to show that $\underset{{\pmb x}\in K^d}{\min} \, p \geq \underset{{\pmb x}\in K^d}{\min} \, \overline{p}$ and then
  \begin{equation*}
     \underset{{\pmb x}\in K^d}{\min}\, \overline{p} \leq \overline{p}(\pmb {a})=\frac{1}{d!}\sum p(a_1,a_1,\ldots,a_1)=p({\pmb a})=\underset{{\pmb x}\in K^d}{\min}\, p. \qedhere
  \end{equation*}
\end{proof}

\begin{theorem}
  Let $p\in \mathbb{R}[\pmb x]$ and $K$ be a closed subset of $\R$. If $\underset{{\pmb x}\in K^d}{\min} \,p = \underset{{\pmb x}\in K^d}{\min} \,\overline{p}$, then $p$ and $\overline{p}$ have minima at the same points. Moreover, if they have a minimum at $\pmb{c}=(c_1,\ldots,c_d)$ then $p$ and $\overline{p}$ have minima at all points in $\{(c_{\pi(1)},\ldots,c_{\pi(d)}):\pi \in \mathcal{S}\}$.
 %  \textcolor{red}{What is $K$? It is not clear!!!}
\end{theorem}

\begin{proof}
  Let $\overline{p}$ have a minimum at $\pmb{c}=(c_1,\ldots,c_d)$, then
  \begin{equation*}
    \underset{{\pmb x}\in K^d}{\min} \, \overline{p}=\overline{p}(\pmb{c})=\frac{1}{d!}\sum_{\pi \in \mathcal{S}} p(c_{\pi(1)},\ldots,c_{\pi(d)}) \geq \underset{{\pmb x}\in K^d}{\min} \, p.
  \end{equation*}
  Since $\underset{{\pmb x}\in K^d}{\min} \, p = \underset{{\pmb x}\in K^d}{\min} \, \overline{p}$, we see that
   \begin{equation*}
    p(c_{\pi(1)},\ldots,c_{\pi(d)})=\underset{{\pmb x}\in K^d}{\min} \, p, \ \ \text{ for all }   \pi \in \mathcal{S};
   \end{equation*}
 consequently, $p$ and $\overline{p}$ have minima at $(c_{\pi(1)},\ldots,c_{\pi(d)})$, for all $\pi \in \mathcal{S}$.
\end{proof}

\section{$[a,b]$-moment sequences}

Recall that a real sequence $\alpha=\lp \ap_n \rp_{n\in \N_0}$ is a \emph{$[a,b]$-moment sequence} if there exists a nonnegative measure $\mu$ on $[a,b]$ such that
\begin{equation*}
\alpha_n = \int_{a}^{b} x^n d\mu (x) \quad \text{for all } n\in \mathbb{N}_0.
\end{equation*}
A moment sequence supported on $[0,1]$ is usually called a \emph{Hausdorff moment sequence}.
In this section, we utilize the same strategy in Section 2 and will verify similar results for Hausdorff moment sequences.

\begin{theorem}[\cite{book:momentproblem}, Theorem 3.13] \label{thm:haus}
A sequence $\alpha=(\alpha_n)_{n\in \N_0}$  is a $[a,b]$-moment sequence if and only if
$$H_m(\alpha) \geq 0 \ \ \text{and} \ \ H_m\left((a+b)E \alpha - E(E \alpha) - ab \alpha\right) \geq 0 \ \ \text{for all} \ \ m\in \N_0.$$
\end{theorem}

Our first result for $[a,b]$-moment sequences are following.

\begin{theorem}\label{thm:mainHM1}
Let $p \in \mathbb{R}[\pmb x]$ such that $p(\pmb x)\geq 0$ on $[a,b]^d$ and
let $\alpha = ( \alpha_n )_{n \in \N_0}$ be a $[a,b]$-moment sequence.
Then the followings hold.
\begin{enumerate}[(i)]
  \item If $a \geq 0$, then $T_p(\alpha)$ is $[a^d,b^d]$-moment sequence.
  \item If $b \leq 0$, then  $T_p(\alpha)$ is $[-|a|^d,-|b|^d]$-moment sequence.
  \item If $a \leq 0 \leq b$ then $T_p(\alpha)$ is $[-|a|^d,b^d]$-moment sequence.
  \item If $-1 \leq a \leq 0 \leq b \leq 1$ then $T_p(\alpha)$ is $[a,b]$-moment sequence.
\end{enumerate}
\end{theorem}

\begin{proof}
(i) Since $\alpha$ is $[a,b]$-moment sequence, there exists a positive Borel measure $\mu$, with $\supp (\mu) \subset [a,b]$ such that
\begin{equation*}
\alpha_n= \int_a^b x^n d\mu \quad \text{for all }n\in \mathbb{N}_0.
\end{equation*}
Then we have that $H_m(\tilde{\alpha}) \geq0$ for all $m \in \N_0$ by the proof of Theorem \ref{thm:mainSM}. Let $\tilde{\alpha}=T_p(\alpha)$ be defined as \eqref{def:Tp}. To use Theorem \ref{thm:haus}, observe that for each $m \in \N_0$ and for all $\xi_0,\ldots,\xi_m \in \C$, the quadratic form of $(a+b)E\tilde{\alpha}-E(E\tilde{\alpha})-ab\tilde{\alpha}$ is
\begin{align*}
\sum_{j=0}^{m}& \sum_{k=0}^{m} \left[(a+b)\tilde{\alpha}_{j+k+1} - \tilde{\alpha}_{j+k+2} - ab\tilde{\alpha}_{j+k} \right] \xi_{j} \xi_{k} \\
   =&\int_a^b \cdots \int_a^b \left( \sum_{j=0}^{m}   (x_1\cdots x_d)^{j} \xi_{j} \right)^2  (x_1 \cdots x_d)f(x_1\cdots x_d) p(\pmb x)  d\mu \cdots d\mu,
\end{align*}
where $f(t):=-(t-a^d)(t-b^d)$. For $a \leq x_k \leq b$ ($k=1,\ldots,d$), we know that  $f(x_1\cdots x_d) \geq 0$ on $\left[a^d,b^d\right]$. Hence $T_p(\alpha)$ is $\left[a^d,b^d\right]$-moment sequence.
The other statements can be proved in a similar way.
\end{proof}

By Theorem \ref{thm:mainHM1}, Theorem \ref{thm:mainSM} and Theorem \ref{thm:Ln} hold for the Hausdorff moment property.

\begin{corollary}
If $p \in \mathbb{R}[\bm{x}]$ is nonnegative on $[0,1]^d$, then $T_p$ has the Hausdorff moment property.
\end{corollary}

\begin{corollary}
  The hypothesis are same as in Theorem \ref{thm:Ln}. If $\alpha=(\alpha_n)_{n\in \N_0}$ is a Hausdorff moment sequence, then so is \eqref{ddd}.
\end{corollary}

\begin{example}
Using the fact that $p(x,y)=\frac{1}{2} (x^r - y^r)(x^s - y^s)$ is nonnegative on $[0,\infty)^2$ for any $r,s\in \N$,
it follows from Theorem \ref{thm:mainHM1} that $T_p(\alpha)=(\alpha_{n+r+s}\alpha_n - \alpha_{n+r} \alpha_{n+s})_{n \in \N_0}$ is a $[0,1]$-moment sequence for any $[0,1]$-moment sequence $\alpha=(\alpha_n)_{n\in \N_0}$.
It is shown in \cite{CNY20} that the Catalan numbers $C_n$,  the shifted Catalan numbers $C_{n+1}$,
the central binomial coefficients $\binom{2n}{n}$, and the Fine numbers $F_n$
are all $[0,4]$-moment sequences.
So, by Theorem \ref{thm:mainHM1} the following all  become $[0,16]$-moment sequences.
\begin{itemize}
\item[(i)] $(C_{n+r+s}C_n - C_{n+r} C_{n+s})_{n \in \N_0}$
\item[(ii)] $\left(\binom{2(n+r+s)}{(n+r+s)}\binom{2n}{n} - \binom{2(n+r)}{n+r}\binom{2(n+s)}{n+s} \right)_{n \in \N_0}$
\item[(iii)] $(F_{n+r+s}F_n - F_{n+r} F_{n+s})_{n \in \N_0}$
\end{itemize}
Furthermore, since (restricted) hexagonal numbers $h_n$ is a $[1,5]$-moment sequence,
 $(h_{n+r+s}h_n - h_{n+r} h_{n+s})_{n \in \N_0}$ is a $[1,25]$-moment sequence.
 Since
the central Delannoy numbers $D_n$,
the large Schr\'{o}der numbers $r_n$, and
the little Schr\'{o}der numbers $S_n$
are $[3-2\sqrt{2},3+2\sqrt{2}]$-moment sequences,
the image of given sequences under $T_p$ are $[17-12\sqrt{2},17+12\sqrt{2}]$-moment sequences.
\end{example}

\bigskip

\noindent\textit{Acknowledgment.} The first-named author was supported by Basic Science
Research Program through the National Research Foundation of Korea(NRF)
funded by the Ministry of Education (NRF-2020R1I1A1A01068230).
The second-named author was supported by the National Research Foundation of Korea(NRF) grant funded by the Korea government(MSIT) (2020R1C1C1A01009185).
The third-named author  was supported by the National Research Foundation of Korea (NRF) grant funded by the Korea government (MSIT) (2020R1F1A1A01070552).

\end{document}